\documentclass{amsart}
\usepackage{amssymb,mathrsfs,mathptmx,txfonts,enumerate,tensor,amsaddr,mathtools}
\mathtoolsset{showonlyrefs,showmanualtags}

\newtheorem{thm}{Theorem}[section] \newtheorem{lemma}[thm]{Lemma}
\newtheorem{cor}[thm]{Corollary} \newtheorem{prop}[thm]{Proposition}
\theoremstyle{definition} \newtheorem{definition}[thm]{Definition}
\newtheorem{remark}[thm]{Remark} 
\newtheorem{notation}[thm]{Notation}
\newtheorem{example}[thm]{Example} 
 
\theoremstyle{remark}

\newcommand\mstrut{{\phantom{.}}}\newcommand\bull{{\!\hbox{\bf .}}}
\newcommand\newdot{{\kern.8pt\cdot\kern.8pt}}
\newcommand\nbull{{\kern.8pt\raise1.5pt\hbox{\bf .}\kern.8pt}}

\font\sevenrm=cmr7

\newcommand\A{\mathbb{A}}\newcommand\E{\mathbb{E}}
\renewcommand\L{L}
\newcommand\R{\mathbb{R}}
\renewcommand\P{\mathbb{P}}

 \newcommand\SC{\mathscr C}
 \newcommand\SF{\mathscr F}
 
 \newcommand\SR{\mathscr R}
 \newcommand\SL{\mathscr L}

\newbox\ovlbox
\def\ovl#1{\setbox\ovlbox\hbox{$#1$}\rlap{\kern.5\wd\ovlbox
    \ifx#1\SF\else\ifx#1\nabla\kern-2pt\else\ifx#1\CA\else
    \ifx#1J\kern-.5pt\else\ifx#1\Psi\kern-2pt
    \else\ifx#1A\kern-1pt\else\ifx#1f\kern-1pt\else\ifx#1\ast\kern-2pt\else
    \ifx#1Y\kern-2pt\else\ifx#1\omega\kern-2pt
    \else\kern-1.5pt\fi\fi\fi\fi\fi\fi\fi\fi\fi\fi
    $\overline{\hbox to4pt{\hss$\phantom{#1}$\hss}}$\hss}#1}

\catcode`\@=11
\def\acong{\mathrel{\mathpalette\@avereq\sim}} % isomorphic sign
\def\@avereq#1#2{\lower.5\p@\vbox{\baselineskip\z@skip\lineskip-.5\p@
    \ialign{$\m@th#1\hfil##\hfil$\crcr#2\crcr\longrightarrow\crcr}}}
\def\mequal{\mathrel{\mathpalette\@mvereq{\hbox{\sevenrm m}}}}
\def\@mvereq#1#2{\lower.5\p@\vbox{\baselineskip\z@skip\lineskip1.5\p@
    \ialign{$\m@th#1\hfil##\hfil$\crcr#2\crcr=\crcr}}}
\def\relop#1#2#3{\mathrel{\mathop{\kern\z@ #2}\limits^{#1}_{#3}}}
\newcommand\rightinj{\lhook\joinrel\rightarrow}
\newcommand\1{\hbox{\kern.375em\vrule height1.57ex depth-.1ex
    width.05em\kern-.375em \rm 1}}

\def\partr#1{/\!/_{\!#1}^{\phantom{.}}}
\def\invpartr#1{/\!/_{\!#1}^{-1}}

\def\vol{{\operatorname{vol}}} 
\def\Ric{{\operatorname{Ric}}}
\def\End{{\operatorname{End}}}

\DeclareMathOperator{\divv}{div} \DeclareMathOperator{\trace}{trace}
\DeclareMathOperator{\tra}{tr} \DeclareMathOperator{\trans}{}

\def\mathpal#1{\mathop{\mathchoice{\text{\rm #1}}%
    {\text{\rm #1}}{\text{\rm #1}}%
    {\text{\rm #1}}}\nolimits} \def\id{{\mathpal{id}}}
\def\boxit#1{\vbox{\hrule\hbox{\vrule\kern3pt
      \vbox{\kern3pt#1\kern3pt}\kern3pt\vrule}\hrule}}
\def\Aut{{\mathpal{Aut}}} \def\OM{\mathpal{O}(M)}
\def\On{\mathpal{O}(n)} 

\let\oldtocsection=\tocsection \let\oldtocsubsection=\tocsubsection
\renewcommand{\tocsection}[2]{\hspace{0em}\oldtocsection{#1}{#2}}
\renewcommand{\tocsubsection}[2]{\hspace{2em}\oldtocsubsection{#1}{#2}}

\numberwithin{equation}{section}
  
\begin{document}
\title[Derivative and divergence formulae]{Derivative and divergence formulae\\
  for diffusion semigroups} \author{Anton Thalmaier {\rm and} James
  Thompson}
\address{Mathematics Research Unit, University of Luxembourg} \email{anton.thalmaier@uni.lu}
\email{james.thompson@uni.lu}

\begin{abstract}
  For a semigroup $P_t$ generated by an elliptic operator on a smooth
  manifold $M$, we use straightforward martingale arguments to derive
  probabilistic formulae for $P_t(V(f))$, not involving derivatives of
  $f$, where $V$ is a vector field on $M$. For non-symmetric
  generators, such formulae correspond to the derivative of the heat
  kernel in the \textit{forward} variable. As an application, these
  formulae can be used to derive various \textit{shift-Harnack}
  inequalities.
\end{abstract}

\keywords{Diffusion semigroup, Heat kernel, Gradient estimate, Harnack inequality, 
  Ricci curvature} \subjclass[2010]{58J65, 60J60, 53C21} \date\today

\maketitle

\section*{Introduction}

For a Banach space $E$, $e \in E$ and a Markov operator $P$ on
$\mathcal{B}_b(E)$, it is known that certain estimates on
$P(\nabla_e f)$ are equivalent to corresponding \textit{shift-Harnack}
inequalities. This was proved by F.-Y. Wang in \cite{wang}. For
example, for $\delta_e \in (0,1)$ and
$\beta_e \in C((\delta_e,\infty)\times E;[0,\infty))$, he proved that
the derivative-entropy estimate
\begin{equation}\label{eq:estone}
  \big|P(\nabla_e f)\big| \leq \delta\big( P(f \log f) - (Pf)\log Pf\big) + \beta_e(\delta,\newdot)Pf
\end{equation}
holds for any $\delta \geq \delta_e$ and positive $f \in C_b^1(E)$ if
and only if the inequality
\begin{equation}
  (Pf)^p \leq \left(P(f^p(re+\newdot))\right)\,\exp\left(\int_0^1 \frac{pr}{1+(p-1)s} \,\beta_e\left(\frac{p-1}{r+r(p-1)s},\newdot + sre\right)ds\right)
\end{equation}
holds for any $p \geq 1/(1-r\delta_e)$, $r \in (0,1/\delta_e)$ and
positive $f \in \mathcal{B}_b(E)$. Furthermore, he also proved that if
$C\geq 0$ is a constant then the $L^2$-derivative inequality
\begin{equation}\label{eq:esttwo}
  \big|P(\nabla_e f)\big|^2 \leq CPf^2
\end{equation}
holds for any non-negative $f \in C^1_b(E)$ if and only if the
inequality
\begin{equation}
  Pf \leq P\big(f(\alpha e+ {\newdot})\big) + |\alpha|\sqrt{CPf^2}
\end{equation}
holds for any $\alpha \in \R$ and non-negative $f \in \mathcal{B}_b (E)$.
The objective of this article is to find probabilistic formulae for
$P_T(V(f))$ from which such estimates can be derived, for the case in
which $P_T$ is the Markov operator associated to a non-degenerate
diffusion $X_t$ on a smooth, finite-dimensional manifold $M$, and $V$
a vector field.

In Section \ref{sec:one} we suppose that $M$ is a Riemannian manifold
and that the generator of $X_t$ is $\Delta +Z$, for some smooth vector
field $Z$. Any non-degenerate diffusion on a smooth manifold induces a
Riemannian metric with respect to which its generator takes this
form. The basic strategy is then to use the relation
$V(f) = \divv (fV) -f \divv V$ to reduce the problem to finding a
suitable formula for $P_T(\divv (fV))$. Such formulae have been given in
\cite{DriverThal_2002} and \cite{ElworthyLi2} for the case $Z=0$, which we extend to the general case with Theorem \ref{thm:MainTheorem}. In doing so, we do not
make any assumptions on the derivatives of the curvature tensor, as occurred in \cite{CruzeiroZhang}. For an adapted
process $h_t$ with paths in the Cameron-Martin space
$L^{1,2}([0,T];\R)$, with $h_0=0$ and $h_T =1$ and under
certain additional conditions, we obtain the formula
\begin{align}
  &P_T\left(V(f)\right)(x)= -\E\big[f(X_T(x))\,(\divv V)(X_T(x))\big] \\
  &\hphantom{{}={}} +\frac12\E\left[f(X_T(x))\,\Big\langle V(X_T(x)), \partr T \Theta_T \int_0^T \left(\dot h_t-(\divv Z)(X_t(x))h_t\right)\,\Theta_t^{-1}dB_t\Big\rangle\right]
\end{align}
where $\Theta$ is the $\Aut(T_xM)$-valued process defined by the
pathwise differential equation
$$\frac{d}{dt}\Theta_t=- \invpartr t \left(\Ric^\sharp+
  (\nabla_{\bull}Z)^{\ast}-\divv Z\right) \partr t \Theta_t$$
with $\Theta_0=\id_{T_xM}$. Here $//_t$ denotes the stochastic
parallel transport associated to $X_t(x)$, whose antidevelopment to
$T_xM$ has martingale part $B$. In particular, $B$ is a diffusion on
$\R^n$ generated by the Laplacian; it is a standard Brownian
motion sped up by $2$, so that $dB^i_t dB^j_t = 2 \delta_{ij} \,dt$.
Choosing $h_t$ explicitly yields a formula from which estimates then can
be deduced, as described in Subsection \ref{subsec:one_SH}.

The problem of finding a suitable formula for $P_T(V(f))$ is dual to
that of finding an analogous one for $V(P_Tf)$. A formula for the latter is
called the Bismut formula \cite{bismut} or the Bismut-Elworthy-Li
formula, on account of \cite{ElworthyLi}. We provide a brief proof of
it in Subsection \ref{ss:formdif}, since we would like to compare it
to our formula for $P_T(V(f))$. Our approach to these formulae is based
on martingale arguments; integration by parts is done at the level of
local martingales. Under conditions which assure that the local
martingales are true martingales, the wanted formulae are then
obtained by taking expectations. They allow for the choice of a finite
energy process. Depending on the intended type, conditions are imposed either
on the right endpoint, as in the formula for $P_T(V(f))$, or the left
endpoint, as in the formula for $V(P_Tf)$. The formula for $P_T(V(f))$
requires non-explosivity; the formula for $V(P_Tf)$ does not. From the
latter can be deduced Bismut's formula for the logarithmic derivative
in the \textit{backward} variable $x$ of the heat kernel $p_T(x,y)$
determined by
$$(P_Tf)(x)=\int_M f(y)p_T(x,y)\,\vol(dy),\quad f\in C_b(M).$$ 
From our formula for $P_T(V(f))$ can be deduced the following formula for the
derivative in the \textit{forward} variable $y$:
\begin{equation}
  (\nabla \log p_T(x,\newdot))_y = -\frac12\E\left[{\partr T} \Theta_T \int_0^T \left(\dot h_t-(\divv Z)(X_t(x))h_t\right)\Theta_t^{-1} dB_t \big\vert \, X_T(x) = y \right].
\end{equation}

In Section \ref{sec:two} we consider the general case in which $M$ is
a smooth manifold and $X_t$ a non-degenerate diffusion solving a
Stratonovich equation of the form
\begin{equation}
  d X_t = A_0(X_t)\,dt+ A(X_t)\circ dB_t.
\end{equation}
We denote by $TX_t$ the derivative (in probability) of the solution
flow. Using a similar approach to that of Section \ref{sec:one}, and a variety of geometric objects naturally associated to the equation, we obtain, under certain conditions, the formula
\begin{equation}
  \begin{split}
    &P_T(V(f))= -\sum_{i=1}^m \E\left[ f(X_T)\,A_i\langle V,A_i\rangle (X_T)\right]\\
    &+\frac12\E\left[ f(X_T) \,\bigg \langle V,\,\Xi_T
      \int_0^T {\Xi^{-1}_{t}} \left(\left(\dot{h}_t-(\trace
          \hat{\nabla}A_0)(X_t)h_t\right)A(X_t)\,dB_t + 2h_t
        A_0^A \,dt\right)\bigg \rangle\right]
  \end{split}
\end{equation}
with
\begin{equation}
  \begin{split}
    \Xi_t =\text{ }& TX_{t} - TX_{t} \int_0^t TX_{s}^{-1} \left(\left((\hat{\nabla}A_0)^\ast + \hat{\nabla}A_0+ \trace \hat{\nabla}A_0\right)(\Xi_s)\right)ds,\\
    A_0^A=\text{ }&\sum_{i=1}^m\left((\hat{\nabla}A_0)^{ \ast}
      +\hat{\nabla} A_0\right)\left(\breve{T}(\newdot,A_i)^\ast
      (A_i)\right)+\big[A_0,\breve{T}(\newdot,A_i)^\ast (A_i)\big],
  \end{split}
\end{equation}
where the operators $\hat{\nabla} A_0$ and
$\breve{T}(\newdot,A_i)$ are given at each $x \in M$ and $v \in T_xM$
by
\begin{equation}
  \begin{split}
    \hat{\nabla}_v A_0 =\text{ }& A(x)\left(d(A^\ast(\newdot) A_0(\newdot))_x (v) - (dA^\ast)_x(v,A_0)\right),\\
    \breve{T}(v,A_i)_x =\text{ }& A(x) (dA^\ast)_x(v,A_i).
  \end{split}
\end{equation}
This formula has the advantage of involving neither parallel transport nor
Riemannian curvature, both typically difficult to calculate in terms
of $A$.

\section{Intrinsic Formulae}\label{sec:one}

\subsection{Preliminaries}\label{ss:prelim}

Let $M$ be a complete and connected $n$-dimensional Riemannian
manifold, $\nabla$ the Levi-Civita connection on $M$ and
$\pi\colon\OM \to M$ the orthonormal frame bundle over $M$. Let
$E\to M$ be an associated vector bundle with fibre $V$ and structure
group $G=\On$.  The induced covariant derivative
\begin{align*}
  \nabla\colon\Gamma(E)\to\Gamma(T^*M\otimes E)
\end{align*}
determines the so-called \textit{connection Laplacian} (or
\textit{rough Laplacian}) $\square$ on $\Gamma(E)$,
$$\square a=\trace\nabla^2 a.$$ Note that
$\nabla^2 a\in \Gamma(T^*M\otimes T^*M\otimes E)$ and
$(\square a)_x = \sum_i\nabla^2 a (v_i,v_i)\in E_x$ where $v_i$ runs
through an orthonormal basis of $T_x M$. For $a,b\in\Gamma(E)$ of
compact support it is immediate to check that
$$\langle\square a,b\rangle_{L^2(E)}=-\langle\nabla a,\nabla
b\rangle_{L^2(T^*M\otimes E)}.$$
In this sense we have $\square=-\nabla^*\nabla$. Let $H$ be the
horizontal subbundle of the $G$-invariant splitting of $T\OM$ and
\begin{equation*}\label{Zci}
  h\colon\,\pi^\ast TM\acong H\rightinj T\OM
\end{equation*}
the \textit{horizontal lift} of the $G$-connection; fibrewise this
bundle isomorphism reads as
$$h_u\colon\,T_{\pi(u)}M\acong H_u,\quad u\in\OM.$$ In terms of the
\textit{standard horizontal vector fields} $H_1,\ldots,H_n$ on
$\OM$, $$H_i(u):=h_u(ue_i),\quad u\in\OM,$$ \textit{Bochner's
  horizontal Laplacian} $\Delta^{\text{hor}}$, acting on smooth
functions on $\OM$, is given as
$$\Delta^{\text{hor}}=\sum_{i=1}^{n} H_i^2.$$ To formulate the
relation between $\square$ and $\Delta^{\text{hor}}$, it is convenient
to write sections $a\in \Gamma(E)$ as equivariant functions
$F_a\colon\OM\to V$ via $F_a(u)=u^{-1}a_{\pi(u)}$ where we read
$u\in\OM$ as an isomorphism $u\colon V\acong E_{\pi(u)}$. Equivariance
means that $$F_a(ug)=g^{-1}F_a(u),\quad u\in\OM,\ g\in G=\On.$$

\begin{lemma}[see \cite{K-N}, p.\,115]
  For $a\in \Gamma(E)$ and $F_a$ the corresponding equivariant
  function on $\OM$, we have
  $$(H_i F_a)(u)=F_{\nabla_{ue_i}a}(u),\quad u\in\OM.$$
  Hence $$\Delta^{\text{\rm hor}}F_a=F_{\square a},$$ where as above
  $$\square\colon \Gamma(E)\overset{\nabla}{\longrightarrow}\Gamma(T^*
  M\otimes E) \overset{\nabla}{\longrightarrow}\Gamma(T^* M\otimes T^*
  M\otimes E) \overset{\text{\rm trace}}{\longrightarrow} \Gamma
  (E).$$
\end{lemma}

\begin{proof}
  Fix $u\in \OM $ and choose a curve $\gamma$ in $M$ such that
  $\gamma(0)=\pi(u)$ and $\dot\gamma=ue_i$.  Let $t\mapsto u(t)$ be
  the horizontal lift of $\gamma$ to $\OM$ such that $u(0)=u$.  Note
  that $\dot{u}(t)=h_{u(t)}\left(\dot{\gamma}(t)\right)$, and in
  particular $\dot{u}(0)=h_u(ue_i)=H_i(u)$.  Hence, denoting the
  parallel transport along $\gamma$ by
  $\partr\varepsilon=u(\varepsilon)u(0)^{-1}$, we get
  \begin{align*}
    F_{\nabla_{ue_i}a}(u)&=u^{-1}\left(\nabla_{ue_i}a\right)_{\pi(u)}\\
                         &=u^{-1}\lim_{\varepsilon \downarrow 0}\frac{\invpartr\varepsilon 						   a_{\gamma(\varepsilon)}-a_{\gamma(0)}}{\varepsilon}\\
                         &=\lim_{\varepsilon\downarrow0} \frac{u(\varepsilon)^{-1}								   a_{\gamma(\varepsilon)}-u(0)^{-1}a_{\gamma(0)}}{\varepsilon}\\
                         &=\lim_{\varepsilon\downarrow0} \frac{F_a (u(\varepsilon))-							   F_a(u(0))}{\varepsilon}\\
                         &=(H_i)_u F_a\\
                         &= (H_iF_a)(u).
  \end{align*}
\end{proof}

Now consider diffusion processes $X_t$ on $M$ generated by the
operator
\begin{equation}
  \SL=\Delta + Z
\end{equation}
where $Z\in\Gamma(TM)$ is a smooth vector field.
% Assume, for simplicity, that $X_t$ is non-explosive.
Such diffusions on $M$ may be constructed from the corresponding
horizontal diffusions on $\OM$ generated by
$$\Delta^{\text{\rm hor}}+\bar Z$$ where the vector field $\bar Z$ is
the horizontal lift of $Z$ to $\OM$, i.e. $\bar Z_u=h_u(Z_{\pi(u)})$,
$u\in\OM$. More precisely, we start from the Stratonovich stochastic
differential equation on $\OM$,
\begin{equation}\label{Eq:StratoLaplianonOM}
  dU_t=\sum_{i=1}^{n}H_i(U_t)\circ dB^i_t+\bar Z(U_t)\,dt,\quad U_0=u\in\OM
\end{equation}
where $B_t$ is a Brownian motion on $\R^n$ sped up by $2$,
that is $dB^i_t dB^j_t = 2 \delta_{ij} \,dt$. Then for
$X_t=\pi (U_t)$, the following equation holds:
\begin{equation}\label{Eq:StratoLaplianonM}
  dX_t=\sum_{i=1}^{n}U_t e_i\circ dB^i_t+Z(X_t)\,dt,\quad X_0=x:=\pi u.
\end{equation}
The Brownian motion $ B$ is the martingale part of the anti-development
$\int_U\vartheta$ of $X$, where $\vartheta$ denotes the canonical
$1$-form $\vartheta$ on $\OM$,
i.e. $$\vartheta_u(e)=u^{-1}e_{\pi(u)},\quad e\in T_u\OM.$$ In
particular, for $F\in C^\infty(\OM)$, resp. $f\in C^\infty(M)$, we
have
\begin{align}
  d(F\circ U_t)&=\sum_{i=1}^{n}(H_i F)(U_t)\circ dB^i_t+(\bar ZF)(U_t)\,dt\notag \\
               &=\sum_{i=1}^{n} (H_i F)(U_t)\,dB^i_t+ \left(\Delta^{\text{hor}}+\bar Z\right)(F)(U_t)\,dt,\label{Eq:ItoOM} \\
  \intertext{respectively}
  d(f\circ X_t)&=\sum_{i=1}^{n} (df)(U_t e_i)\circ dB^i_t+(Zf)(X_t)\,dt\notag\\
               &=\sum_{i=1}^{n} (df)(U_t e_i)\,dB^i_t+\left(\Delta+Z\right)(f)(X_t)\,dt.\label{Eq:ItoM}
\end{align}
Typically, solutions to \eqref{Eq:StratoLaplianonM} are defined up to
some maximal lifetime $\zeta(x)$ which may be finite. Then we have,
almost surely,
$$\big\{\zeta(x)<\infty\big\}\subset \big\{X_t\to\infty\text{ as $t\uparrow\zeta(x)$}\big\}$$
where on the right-hand side, the symbol $\infty$ denotes the point at
infinity in the one-point compactification of $M$. It can be shown
that the maximal lifetime of solutions to equation
\eqref{Eq:StratoLaplianonOM} and to \eqref{Eq:StratoLaplianonM}
coincide, see e.g.~\cite{Shigekawa:82}.

In case of a non-trivial lifetime the subsequent stochastic equations
should be read for $t<\zeta(x)$.

\begin{prop}\label{Prop:covarIto}
  Let $\partr t\colon{E_{X_0}}\to{E_{X_t}}$ be parallel transport in
  $E$ along $X$, induced by the parallel transport on $M$,
  $${\partr t=U_tU_0^{-1}}\colon{T_{X_0}M}\to{T_{X_t}M}.$$ Then, for
  $a \in\Gamma(E)$, we have
  \begin{equation}\label{Eq:ItoRoughLapl}
    d\left(\invpartr t a(X_t)\right)=\sum_{i=1}^{n} \invpartr t \left(\nabla_{U_t e_i}a \right)\circ dB^i_t +\invpartr t \left(\nabla_Z a\right)(X_t)\,dt,
  \end{equation}
  respectively in It\^o form,
  \begin{equation}\label{Eq:StratoRoughLapl}
    d\left(\invpartr t a(X_t)\right)=\sum_{i=1}^{n}\invpartr t\left(\nabla_{U_t e_i}a\right)dB^i_t +\invpartr t\left(\square a+\nabla_Za \right)(X_t)\,dt.
  \end{equation}
  More succinctly, the last two equations may be written as
  $$d\left(\invpartr t a(X_t)\right)=\invpartr t\nabla_{\circ
    dX_t}a,$$ respectively
  $$d\left(\invpartr t a(X_t)\right)=\invpartr t
  \nabla_{dX_t}\alpha+\invpartr t(\square a)(X_t)\,dt.$$
\end{prop}

\begin{proof}
  We have $\invpartr t a(X_t)=U_0 U_t^{-1} a(X_t)=U_0 F_a(U_t)$. It is
  easily checked that $ \bar Z F_a=F_{\nabla_Za}$.  Thus, we obtain
  from equation \eqref{Eq:ItoOM}
  \begin{align*}
    dF_a(U_t)&=\sum_{i=1}^{n} (H_i F_a)(U_t)\,dB^i_t+\left( \Delta^{\text{hor}}F_a+\bar Z F_a\right)(U_t) \,dt\\
             &=\sum_{i=1}^{n} \left(F_{\nabla_{U_t e_i}a}\right)(U_t)\,dB^i_t+\left( F_{\square a}+F_{\nabla_Z a}\right)(U_t) \,dt\\
             &=\sum_{i=1}^{n}
               U^{-1}\left(\nabla_{U_t
               e_i}a\right)(X_t)\,dB^i_t+
               U^{-1}_t\left(\square
               a+\nabla_Za\right)(X_t)
               \,dt.
                                 \end{align*}
                               \end{proof}

\begin{cor}
  Fix $T>0$ and let $a_t\in\Gamma(E)$ solve the equation
  $$\frac{\partial}{\partial t}a_t=\square a_t +\nabla_Za_t\quad
  \text{on}\ [0,T]\times M.$$ Then
  \begin{equation*}\label{Eq:RoughLaplMart}
    \invpartr t a^{\mstrut}_{T-t}\left(X_t\right), \quad 0\leq t< T\wedge\zeta(x),
  \end{equation*}
  is a local martingale.
\end{cor}

\begin{proof}
  Indeed we have
  $$d(\invpartr t a_{T-t}\left(X_t\right))\mequal\invpartr
  t\underbrace{\left(\square
      a_{T-t}+\nabla_Za_t+\frac{\partial}{\partial
        t}a_{T-t}\right)}_{=0}\left(X_t\right)\,dt=0,$$
  where $\mequal$ denotes equality modulo differentials of local
  martingales.
\end{proof}

We are now going to look at operators $\SL^{\SR}$ on $\Gamma(E)$ which
differ from $\square$ by a zero-order term, in other words,
\begin{equation}\label{Eq:Weitzenbock}
  \square-\SL^{\SR}=\SR\quad \text{where }\SR \in \Gamma(\End E).
\end{equation}
Thus, by definition, the action $\SR_x\colon E_x\to E_x$ is linear for
each $x\in M$.

\begin{example}
  A typical example is $E=\Lambda^p T^*M$ and
  $A^p(M)=\Gamma(\Lambda^p T^*M)$ with $p\geq1$. The
  \textit{de~Rham-Hodge
    Laplacian} $$\Delta^{(p)}=-(d^*d+dd^*):A^p(M)\to A^p(M)$$ then
  takes the form $$\Delta^{(p)} \alpha=\square \alpha-\SR \alpha$$
  where $\SR$ is given by the Weitzenb\"ock decomposition. In the
  special case $p=1$, one obtains
  $\SR\alpha=\Ric(\alpha^{\sharp},\newdot)$ where
  $\Ric\colon TM\oplus TM \to \R$ is the Ricci tensor.
  % , or equivalently $\SR \alpha = \Ric^{\sharp \tra} \alpha$.
\end{example}

\begin{definition}\label{Def:pathwiseQ}
  Fix $x\in M$ and let $X_t$ be a diffusion to $\SL=\Delta+Z$, starting
  at $x$. Let $Q_t$ be the $\Aut(E_x)$-valued process defined by the
  following linear pathwise differential
  equation $$\frac{d}{dt}Q_t=-Q_t\SR_{\partr t},\quad Q_0=\id_{E_x},$$
  where
  $$\SR_{\partr t}:=\invpartr t\circ\SR_{X_t}\circ\partr
  t\in\End(E_x)$$
  and $\partr t$ is parallel transport in $E$ along~$X$.
\end{definition}

\begin{prop}\label{Prop:diffRoughLapl}
  Let $\SL^{\SR}=\square-\SR$ be as in equation \eqref{Eq:Weitzenbock}
  and $X_t$ be a diffusion to $\SL=\Delta+Z$, starting at $x$. Then, for
  any $a\in\Gamma(E)$,
  $$d\left(Q_t \invpartr t a(X_t)\right) =\sum_{i=1}^{n} Q_t \invpartr
  t\left(\nabla_{U_te_i}a\right)\,dB^i_t+Q_t \invpartr t \left(\square
    a+\nabla_Za-\SR a\right)(X_t)\,dt.$$
\end{prop}

\begin{proof}
  Let $n_t:=\invpartr ta(X_t)$. Then
  \begin{align*}
    d(Q_tn_t)&=(dQ_t)\,n_t+Q_t\,dn_t\\
             &=-Q_t\invpartr t \SR_{X_t}\invpartr t n_t\,dt+Q_t\,dn_t\\
             &=-Q_t\invpartr t (\SR a)(X_t)\,dt+Q_t\,dn_t.
  \end{align*}
  The claim thus follows from Proposition \ref{Prop:covarIto}.
\end{proof}

\begin{cor}\label{cor:martpropgen}
  Fix $T>0$ and let $X_t(x)$ be a diffusion to $\SL=\Delta+Z$, starting
  at $x$. Suppose that $a_t$ solves
$$\left\{\begin{aligned} &\frac{\partial}{\partial t}a_t=\left( \square-\SR+\nabla_Z\right)a_t\quad \text{on}\ [0,T]\times M,\\ &a_t|_{t=0}=a \in \Gamma(E). \end{aligned}\right.$$ 
Then
\begin{equation}\label{Eq:locmart}
  N_t:=Q_t\invpartr t a_{T-t}\left(X_t(x)\right),\quad0\leq t< T\wedge\zeta(x),
\end{equation}
is a local martingale, starting at $a_T(x)$. In particular, if
$\zeta(x)=\infty$ and if equation \eqref{Eq:locmart} is a true
martingale on $[0,T]$, we arrive at the formula
$$a_T(x)=\E\big[Q_T \invpartr T a(X_T(x))\big],\quad a\in\Gamma(E).$$
\end{cor}

\begin{proof}
  Indeed, we have
  \begin{equation*}
    dN_t\mequal Q_t\invpartr t \underbrace{\left((\square+\nabla_Z-\SR)a^{\mstrut}_{T-t}+\frac{\partial}{\partial t}a^{\mstrut}_{T-t}\right)}_{=0} (X_t)\,dt=0
  \end{equation*}
  as required. \qedhere
\end{proof}

\begin{remark}
  Note that
  $$\frac{d}{dt}Q_t=-Q_t\SR_{\partr t},\quad \text{with
    $Q_0=\id_{E_x}$},$$ implies the obvious estimate
  $$\|Q_t\|_{\text{op}}\le\exp\left(-\int_0^t
    \underline{\SR}(X_s(x))ds\right)$$
  where
  $\underline{\SR}(x)=\inf\left\{\langle\SR_x v,w\rangle\colon v,w\in
    E_x,\ \|v\|\leq1 \text{ and }\|w\|\leq1\right\}$.
\end{remark}

\subsection{Commutation formulae}

In the sequel, we consider the special case $E=T^*M$. Thus $\Gamma(E)$
is the space of differential $1$-forms on $M$. The results of this
section apply to vector fields as well, by identifying vector fields
$V\in\Gamma(TM)$ and $1$-forms $\alpha\in\Gamma(T^*M)$ via the metric:
$$V\longleftrightarrow
V^\flat,\quad\alpha\longleftrightarrow\alpha^\#.$$
Let $Z\in\Gamma(TM)$ be a vector field on $M$. Then the divergence of
$Z$, denoted by $\divv Z\in\SC^{\infty}(M)$, is defined by
$\divv Z:=\ \trace(v\mapsto\nabla_v Z).$ Therefore
\begin{equation}
  (\divv Z)(x)=\sum_{i=1}^n \langle\nabla_{v_i}X,v_i\rangle
\end{equation}
for any orthonormal basis $\lbrace v_i \rbrace_{i=1}^n$ for
$T_xM$. For compactly supported $f$ we have
\begin{equation}
  \langle Z,\nabla f\rangle_{L^2(TM)}=-\langle \divv Z,f\rangle_{L^2(M)}.
\end{equation}
The adjoint $Z^*$ of $Z$ is given by the relation
$$Z^*f=-Zf-(\divv Z)f,\quad f\in C^\infty(M).$$ If either $f$ or $h$
is compactly supported, this implies
\begin{equation}
  \langle Zf,h\rangle_{L^2(M)}=\langle f,Z^*h\rangle_{L^2(M)}.
\end{equation}
Similarly, for
$\alpha\in\Gamma(T^\ast M)$, we let
$$(\divv\alpha)(x)=\trace\bigl(T_xM\buildrel{\!\!\!\nabla\alpha}\over\longrightarrow
T^\ast_xM\buildrel\#\over\longrightarrow T_xM\bigr).$$
Thus $\divv Y=\divv Y^\flat$ and $\divv\alpha=\divv\alpha^\#$. That
is, if $\delta = d^\ast$ denotes the usual codifferential then
$\divv \alpha = -\delta \alpha$. Finally, we define
\begin{equation*}
  \Ric_Z(X,Y):=\Ric(X,Y)-\langle\nabla_XZ,Y\rangle,\quad X,Y\in\Gamma(TM).
\end{equation*}

\begin{notation}\label{notatation:Ric} 
  For the sake of convenience, we read bilinear forms on $M$, such as
  $\Ric_Z$, likewise as sections of $\End(T^*M)$ or $\End(TM)$, e.g.
  \begin{align}
    \Ric_Z(\alpha)&:=\Ric_Z(\newdot,\alpha^\sharp),\quad\alpha\in T^*M,\\
    \Ric_Z(v)&:=\Ric_Z(v,\newdot)^\sharp,\quad v\in TM.
  \end{align}
  If there is no risk of confusion, we do not distinguish in notation.
  In particular, depending on the context, $(\Ric_Z)_{\partr t}$ may
  be a random section of $\End(T^*M)$ or of $\End(TM)$.
\end{notation}

\begin{lemma}[Commutation rules]\label{Lemma:CommRules}
  Let $Z\in\Gamma(TM)$.
  \begin{enumerate}[\rm(1)]
  \item For the differential $d$, we
    have $$d\big(\Delta+Z\big)=\big(\square-\Ric_Z+\nabla_Z\big)d;$$
  \item for the codifferential $d^*=-\divv$, we have
$$\big(\Delta+Z^*\big)d^*=d^*\big(\square-\Ric_Z^{*}+\nabla^*_Z\big),$$
where the formal adjoint of $\nabla_Z$ (acting on $1$-forms) is
$\nabla^*_Z\alpha=-\nabla_Z\alpha-(\divv Z)\alpha$.
\end{enumerate}
\end{lemma}

\begin{proof}
  Indeed, for any smooth function $f$ we have
  \begin{align*}
    d\big(\Delta+Z\big)f&=d\big(-d^*df+(df)Z\big)\\
                        &=\Delta^{(1)}df+\nabla_Z df+\langle\nabla_\bull Z,\nabla f\rangle\\
                        &=(\square+\nabla_Z )(df)-\Ric_Z(\newdot,\nabla f)\\
                        &=\big(\square-\Ric_Z+\nabla_Z\big)(df).
  \end{align*}
  The formula in (2) is then just dual to $(1)$.
\end{proof}

\subsection{A formula for the differential}\label{ss:formdif}

Now, let $X_t(x)$ be a diffusion to $\Delta+Z$ on $M$, starting at
$X_0(x)=x$, $U_t$ a horizontal lift of $X_t(x)$ to $\OM$ and
$B=U_0\int_U\vartheta$ the martingale part of the anti-development of
$X_t(x)$ to $T_xM$. Let $Q_t$ be the $\Aut(T_x^\ast M)$-valued process
defined by
\begin{equation}
  \frac{d}{dt}Q_t=- Q_t \left(\Ric_Z\right)_{\partr t}
\end{equation}
with $Q_0=\id_{T_x^\ast M}$, let
\begin{equation}
  P_tf(x)=\E\left[\1_{\{t<\zeta(x)\}}f(X_t(x))\right]
\end{equation}
be the minimal semigroup generated by $\Delta+Z$ on $M$, acting on
bounded measurable functions $f$.

Fix $T>0$ and let $\ell_t$ be an adapted process with paths in the
Cameron-Martin space $L^{1,2}([0,T];T_{x}M)$. By Corollary
\ref{cor:martpropgen}
\begin{equation}\label{eq:LocMart_Pdf}
  N_t := Q_t \invpartr t  (dP_{T-t}f),\quad t<T\wedge\zeta(x),
\end{equation}
is local martingale. Therefore
\begin{equation}
  N_t( {\ell}_t) -\int_0^t Q_s \invpartr s  (dP_{T-s}f)( \dot{{\ell}}_s)ds
\end{equation}
is a local martingale. By integration by parts
\begin{equation}
  \int_0^t Q_s \invpartr s  (dP_{T-s}f)(\dot{\ell}_s )ds - \frac12(P_{T-t}f)(X_t(x)) \int_0^t \langle Q_s^{\tra} (\dot{\ell}_s), dB_s\rangle
\end{equation}
is also a local martingale and therefore
\begin{equation}\label{eq:locmartingzero}
  Q_t \invpartr t (dP_{T-t}f) ( {\ell}_t ) -\frac12(P_{T-t}f)(X_t(x))\int_0^t \langle Q_s^{\tra} \dot{\ell}_s ,dB_s\rangle
\end{equation}
is a local martingale, starting at $(dP_Tf)(\ell_0)$. Choosing
$\ell_t$ so that \eqref{eq:locmartingzero} is a true martingale on
$[0,T]$ with $\ell_0=v$ and $\ell_T=0$, we obtain the formula
\begin{equation}\label{eq:locformzeroform}
  (dP_Tf)(v)  = - \frac12\E\left[\1_{\{T<\zeta(x)\}}f(X_T(x)) \int_0^{T} \langle Q_s^{\tra} \dot{\ell}_s ,dB_s\rangle \right].
\end{equation}
For further details, see \cite{Thalmaier97,Thalmaier98}. Denoting by
$p_t(x,y)$ the smooth heat kernel associated to $\Delta +Z$, since
formula \eqref{eq:locformzeroform} holds for all smooth functions $f$
of compact support, it implies Bismut's formula
\begin{equation}
  (d \log p_T(\newdot,y))_x(v) = -\frac12\E\left[ \int_0^{\tau \wedge T} \langle Q_s^{\tra} \dot{\ell}_s, dB_s\rangle \big\vert \, X_T(x) = y \right].
\end{equation}
The argument leading to formula \eqref{eq:locformzeroform} is based on
the fact that the local martingale \eqref{eq:locmartingzero} is a true
martingale. Since the condition on $\ell_t$ is imposed on the left
endpoint, this can always be achieved, by taking $\ell_s=0$ for
$s\geq \tau\wedge T$ where $\tau$ is the first exit time of some
relatively compact neighbourhood of $x$. No bounds on the geometry are
needed; also explosion in finite times of the underlying diffusion can
be allowed.  For the problem of constructing appropriate finite energy
processes $\ell_s$ with the property $\ell_s=0$ for
$s\geq \tau\wedge T$, see \cite{Thalmaier98}, resp.~\cite[Lemma
4.3]{ThalmaierWang11}.

Imposing in \eqref{eq:locmartingzero} however the conditions
$\ell_0=0$ and $\ell_T=v$ would lead to a formula for
$$\E\left[ Q_T\invpartr T(df)_{X_T(x)}(v)\right]$$ not involving
derivatives of $f$, which clearly requires strong assumptions. If the
local martingale \eqref{eq:LocMart_Pdf} is a true martingale, we get
the formula
$$\left(dP_Tf\right)_x(v)=\E\left[ Q_T\invpartr
  T(df)_{X_T(x)}(v)\right].$$
For such a formula to hold, obviously $X_t(x)$ needs to be
non-explosive.

\subsection{A formula for the codifferential}

Recall that, according to Lemma \ref{Lemma:CommRules}, we have
\begin{equation}\label{eq:comrulediv}
  \big(\Delta+Z+\divv Z\big)\divv=\divv\big(\square+\nabla_Z-\Ric_{-Z}^{*}+\divv Z\big).
\end{equation}
For a bounded $1$-form $\alpha$ suppose $\alpha_t$ satisfies
\begin{equation}\label{eq:Semigroupforms}
  \frac{d}{dt}\alpha_t=\left( \square+\nabla_Z-\Ric_{-Z}^{*}+\divv Z\right)\alpha_t
\end{equation}
with $\alpha^{\mstrut}_0 = \alpha$, where $\divv Z$ acts fibrewise as
a multiplication operator, and that $\Theta_t$ is the
$\Aut(T_xM)$-valued process which solves
\begin{equation}\label{Eq:Qtr}
  \frac{d}{dt}\Theta_t=-(\Ric_{-Z}^{*}-\divv Z)_{\partr t}\Theta_t
\end{equation}
with $\Theta_0=\id_{T_xM}$. Here $\Ric_{-Z}^{*}$ is the adjoint to
$\Ric_{-Z}$ acting as endomorphism of $T_xM$, see
Notation~\ref{notatation:Ric}.

\begin{remark} We have $\Theta_t=Q_t^{\tra}$ if we set
  $\SR:=\Ric_{-Z}^{*}-\divv Z\in\End(T^*M)$ and define $Q_t$ via
  Definition \ref{Def:pathwiseQ}.
\end{remark}

\begin{prop}
  Fix $T>0$. Let $X_t(x)$ be a diffusion to $\Delta+Z$ on $M$,
  starting at $x$.
  \begin{enumerate}[\rm(i)]
  \item Then
    \begin{equation}
      (\divv\alpha^\mstrut_{T-t})(X_t(x))\,
      \exp \left(\int_0^t (\divv Z)(X_s(x))\,ds\right)  
    \end{equation}
    is a local martingale, starting at $\divv\alpha^{\mstrut}_T$.
  \item Suppose $h_t$ is an adapted process with paths in
    $L^{1,2}([0,T];\R)$. Then
    \begin{equation}\label{LocMartInt1}
      \begin{split}
        &\divv \alpha^{\mstrut}_{T-t} h_t - \frac12
        \alpha^{\mstrut}_{T-t} \left( \partr t \Theta_t \int_0^t
          \left(\dot{h}_s -(\divv Z)(X_s(x))\,h_s\right) \Theta_s^{-1}
          \invpartr s dB_s\right)
      \end{split}
    \end{equation}
    is a local martingale, starting at $\divv \alpha^{\mstrut}_T h_0$.
  \end{enumerate}
\end{prop}

\begin{proof}
  (i) Taking into account the commutation rule \eqref{eq:comrulediv}
  and the evolution equation \eqref{eq:Semigroupforms} of $\alpha_t$,
  we get
  \begin{equation}\label{Eq:SemigrouponFunctions}
    \begin{split}
      \partial_t \divv \alpha_t &= \divv \partial_t \alpha_t\\
      &= \divv (\square+\nabla_Z-\Ric_{-Z}^{ *}+\divv Z)\alpha_t\\
      &= (\Delta + Z + \divv Z) \divv \alpha_t.
    \end{split}
  \end{equation}
  The claim then follows from It\^{o}'s formula.

  (ii) To verify the second item, set
  \begin{equation}
    \A_t := \exp \left(\int_0^t (\divv Z)(X_s(x))\,ds\right)
  \end{equation}
  and define $\ell_t := \A^{-1}_t h_t$. Using the fact that
  $\alpha^{\mstrut}_{T-t}(\partr t \Theta_t)$ is a local martingale,
  indeed
  \begin{equation}
    d\big(\alpha^{\mstrut}_{T-t}(\partr t \Theta_t)\big)
    =\sum_{i=1}^{n} (\nabla_{U_te_i}\alpha^{\mstrut}_{T-t})(\partr t \Theta_t)\,dB^i_t
  \end{equation}
  we obtain
  \begin{align*}
    (\divv \alpha_{T-t}^\mstrut)&(X_t(x))\A_t \dot{\ell_t}\,dt\\
                                &=\sum_{i=1}^n (\nabla_{U_te_i}^\mstrut \alpha^{\mstrut}_{T-t})\left( U_te_i\right) \A_t \dot{\ell_t}\,dt\\
                                &=\sum_{i=1}^n \left(\invpartr t \nabla_{U_te_i}^\mstrut \alpha^{\mstrut}_{T-t}\right)\left( U_0e_i\right)\A_t \dot{\ell_t}\,dt\\
                                &=\sum_{i=1}^n \left(\nabla_{U_te_i}^\mstrut \alpha^{\mstrut}_{T-t}\right)\big(\partr t\Theta_t\Theta_t^{-1} U_0 e_i \big)\,\A_t \dot{\ell_t}\,dt\\
                                &=\frac12 \Big\langle\sum_{i=1}^n \left(\nabla_{U_te_i}^\mstrut\alpha^{\mstrut}_{T-t}\right)( \partr t \Theta_t)\,dB^i_t\,, \A_t \dot{\ell_t}\Theta_t^{-1} dB_t\Big\rangle\\
                                &\mequal\frac12 d\left(\alpha_{T-t}^\mstrut\Big( \partr t \Theta_t \int_0^t \A_s \dot{\ell_s}\Theta_s^{-1}\,dB_s\Big)\right)
  \end{align*}
  where $\mequal$ denotes equality modulo the differential of a local
  martingale.  By part (i)
  \begin{equation}\label{LocMartDiv}
    n_t:=(\divv\alpha^\mstrut_{T-t})(X_t(x)) \A_t 
  \end{equation}
  is a local martingale and therefore so is
  \begin{equation}\label{LocMartInt}
    n_t\ell_t-\int_0^t n_s\,d\ell_s.
  \end{equation}
  Since
  \begin{equation}
    \A_t \dot{\ell}_t = \dot{h}_t -(\divv Z)(X_t(x))\,h_t
  \end{equation}
  the result follows by substitution.
\end{proof}

\begin{remark} a) Let $D^n$ be an exhausting sequence of $M$ by
  relatively compact open domains.  Following the discussion of
  \cite[Appendix~B]{DriverThal_2002} and \cite[Section
  III.1]{Engel-Nagel:2000} it is standard to show that there is a
  strongly continuous semigroup $P_t^n$ on compactly supported
  $1$-forms~$\alpha$ on $D^n$ generated by
  $L:=\square+\nabla_Z-\Ric_{-Z}^{*}+\divv Z$ with Dirichlet boundary
  conditions. In probabilistic terms,
  $\alpha_t^n(x):=(P_t^n\alpha)(x)$ is easily identified as
$$\alpha_t^n(x)=\E\left[\1_{\{t<\tau^n(x)\}}\,\alpha(\partr t \Theta_t)\right]$$
where $\tau^n(x)$ is the first exit time of $X_t(x)$ from $D^n$, when
started at $x\in D^n$.  As $n\to\infty$, the semigroup $\alpha_t^n$
converges to
\begin{equation}\label{Eq:ReprAlpha}
  \alpha_t(x)=\E\left[\1_{\{t<\zeta(x)\}}\,\alpha(\partr t \Theta_t)\right].
\end{equation}
In particular, $\alpha_t$ solves equation \eqref{eq:Semigroupforms} on
$M$.

b) Formula \eqref{Eq:ReprAlpha} shows that $\alpha_t$ is bounded in
case $\alpha$ is bounded.  Choosing the process~$h$ in
\eqref{LocMartInt1} in such a way that $h_0=1$ but $h_t=0$ for
$t\geq \tau\wedge T$ where $\tau$ is the first exit time of $X_t(x)$
of some relatively compact neighbourhood of $x$, we arrive at the
formula
\begin{equation}\label{Eq:StochReprAlpha}
  (\divv \alpha^{\mstrut}_{T})(x)=-\frac12 \E\left[\1_{\{T<\zeta(x)\}} \alpha\left( \partr T \Theta_T \int_0^T \left(\dot{h}_s -(\divv Z)(X_s(x))\,h_s\right) \Theta_s^{-1} \invpartr s dB_s\right)\right].
\end{equation}
Note that the local formula \eqref{Eq:StochReprAlpha} doesn't require
assumptions, either on the geometry of $M$ or on the drift vector
field $Z$. Indeed, with an appropriate choice of $h$ it is always
possible to make \eqref{LocMartInt1} a true martingale.
\end{remark}

\begin{lemma}\label{lem:truemart}
  Suppose $\Ric_Z$ is bounded below, that ${\Ric} + (\nabla_\bull Z)^\ast$, $\divv Z$ and $\divv \alpha$ are bounded with $h_t$ bounded and
  \begin{equation} {\left(\int_0^{T} |\dot{h}_s|^2
        ds\right)}^{{1}/{2}} \in L^{1+\epsilon}
  \end{equation}
  for some $\epsilon >0$. Then the local martingale
  \eqref{LocMartInt1} is a true martingale.
\end{lemma}

\begin{proof}
  Since ${\Ric}_Z$ is bounded below, the process $X_t$ is non-explosive, by
  \cite[Corollary~2.1.2]{Wangbook}.  In this case we have
  $\alpha_t=\E\big[\alpha(\partr t \Theta_t)\big]$.  From
  equation~\eqref{Eq:SemigrouponFunctions} we see that
$$u(t,x):=(\divv \alpha_t)(x)$$
solves the heat equation
\begin{equation}\label{Eq:HeatEqu}
  \partial_t u= (\Delta + Z + \divv Z) u
\end{equation}
with initial condition $u(0,\newdot) = \divv \alpha$. By means of
equation~\eqref{Eq:StochReprAlpha}, combined with the bound on
$\divv Z$ and the other assumptions, we see that $\divv \alpha_t$ is a
bounded solution to \eqref{Eq:HeatEqu}, which implies
\begin{equation}\label{Eq:Commutation_Formula1}
  \divv \alpha_t = \E \left[ (\divv \alpha)(X_t) \exp\left( \int_0^t (\divv Z)(X_s)\,ds\right)\right]
\end{equation}
for all $t\geq 0$.  Note that our assumptions control the norms of
$\Theta_t$ and $\Theta_t^{-1}$.  Combined with the assumptions on $h$
this proves that \eqref{LocMartInt1} is indeed a true martingale.
\end{proof}

\begin{remark} Equation \eqref{Eq:Commutation_Formula1} shows that
  $\divv$ commutes with the semigroup $P_t^{(1)}\alpha:=\alpha_t$ on
  $1$-forms:
$$ \divv P_t^{(1)}\alpha = P_t^{\divv Z}(\divv\alpha)$$
where
$$P_t^{\rho}f:=\E \left[f(X_t) \exp\left( \int_0^t
    \rho(X_s)\,ds\right)\right]$$
denotes the Feynman-Kac semigroup on functions to $\Delta+Z$ with
scalar potential $\rho$.
\end{remark}

Using the identification of differential forms and vector fields via
the metric, we obtain the following result (which for compact $M$ with $Z=0$ corrects the sign in \cite[Theorem~5.10]{DriverThal_2002}):

\begin{thm}\label{thm:MainTheorem}
  Let $M$ be a Riemannian manifold and $Z$ a smooth vector field
  on~$M$.  Let $X=X(x)$ be a diffusion to $\Delta+Z$ on $M$, starting
  at $X_0(x)=x$, which is assumed to be non-explosive.  Let $T>0$ and
  $h$ be an adapted process with paths in $L^{1,2}([0,T];\R)$ such
  that $h_0=0$ and $h_T=1$, and such that~\eqref{LocMartInt1} is a
  true martingale. Then for all bounded smooth vector
  fields $V$ on~$M$,
  \begin{align*}
    \E\bigl[(\divv V)(X_T(x))\bigr]=\frac12\E\left[\Bigl\langle V(X_T(x)),\partr T \Theta_T \int_0^T \left(\dot h_t-(\divv Z)(X_t(x))h_t\right)\Theta_t^{-1} dB_t\Bigr\rangle\right]
  \end{align*}
  where $\Theta$ is the $\Aut(T_xM)$-valued process defined by the
  following pathwise differential equation:
  $$\frac{d}{dt}\Theta_t=-\Ric_{\partr t}\Theta_t-
  (\nabla_{\bull}Z)^{\ast}_{\partr t}\Theta_t+(\divv Z)\Theta_t$$
  with $\Theta_0=\id_{T_xM}$.
\end{thm}

\begin{cor}\label{cor:formula_cor}
  Suppose $f$ is a bounded smooth function and that $V$ is a bounded smooth vector field with $\divv V$ bounded. Then, under the assumptions of Theorem \ref{thm:MainTheorem}, by using the relation $\divv(fV)=Vf+f\divv V$, we get
  \begin{align}
    &P_T\big(V(f)\big)(x)\\
    =&-\E\big[f(X_T(x))\,(\divv V)(X_T(x))\big] \\
                        &+\frac12\E\left[f(X_T(x))\,\Big\langle V(X_T(x)), \partr T \Theta_T \int_0^T \left(\dot h_t-(\divv Z)(X_t(x))h_t\right)\,\Theta_t^{-1}dB_t\Big\rangle\right]\notag
  \end{align}
  where the right-hand side does not contain any derivatives of $f$.
\end{cor}

\begin{cor}
  Under the assumptions of Theorem \ref{thm:MainTheorem} we have
  \begin{equation}
    \big(\nabla \log p_T(x,\newdot)\big)_y = -\frac12\E\left[{\partr T} \Theta_T \int_0^T \left(\dot h_t-(\divv Z)(X_t(x))h_t\right)\Theta_t^{-1} dB_t \big\vert \, X_T(x) = y \right]
  \end{equation}
  with $\Theta$ given as above.
\end{cor}

\begin{proof}
  By Theorem \ref{thm:MainTheorem}, for all smooth, compactly supported vector fields $V$ we have
  \begin{equation}
    \begin{split}
      &P_T(\divv V)(x) =\\
      &\frac12\int_M \Big\langle V ,\, \E\left[{\partr T} \Theta_T
        \int_0^T \left(\dot
          h_t-(\divv Z)(X_t(x))h_t \right)\Theta_t^{-1} dB_t \big\vert \, X_T(x) = y \right]
      \Big\rangle p_T(x,y) \,dy
    \end{split}
  \end{equation}
  but on the other hand
  \begin{equation}
    \begin{split}
      P_T(\divv V)(x) &= \int_M (\divv V)(y) \,p_T(x,y)\,dy\\
      &=  -\int_M \left(dp_T(x,\newdot)\right)_y\,V(y) \,dy\\
      &= -\int_M \left(d \log p_T(x,\newdot)\right)_y\,V(y)\, p_T(x,y)
      \,dy
    \end{split}
  \end{equation}
  so the result follows.
\end{proof}

\subsection{Shift-Harnack Inequalities}\label{subsec:one_SH}

Suppose $\Ric_Z$ is bounded below, that ${\Ric} + (\nabla_\bull Z)^\ast$ and $\divv Z$ are bounded and that the following formula holds, for all $t >0$, all $f \in C_b^1(M)$ and all bounded vector fields $V$ with $\divv V$ bounded (see Corollary~\ref{cor:formula_cor}):
\begin{equation}
\begin{split}
P_t(V(f))(x)&=-\E\big[f(X_t(x))\,(\divv V)(X_t(x))\big] \\
&\quad +\frac{1}{2}\E\left[f(X_t(x))\,\Big\langle V(X_t(x)), \partr t \Theta_t \int_0^t \bigg[\frac{1}{t}-(\divv Z)(X_r(x))\frac{r}{t}\bigg]\,\Theta_r^{-1}dB_r\Big\rangle\right]\notag.
\end{split}
\end{equation}
Fix $T>0$. Then, by Jensen's inquality (see \cite[Lemma~6.45]{Stroock2000}), there exist constants $c,C_1(T)>0$ such that
\begin{equation}\label{eq:est2}
|P_t(V(f))| \leq \delta \left(P_t (f\log f) - P_tf\log P_tf\right) + {\underbrace{\left(|\divv V|_\infty + \delta c + \frac{C_1(T)}{\delta t}|V|^2_\infty\right)}_{=:\text{ }\alpha_1(\delta,t,V)}} P_tf
\end{equation}
for all $\delta >0$, $t \in (0,T]$ and positive $f \in C^1_b(M)$. Alternatively, by the Cauchy-Schwarz inequality, there exists $C_2(T)>0$ such that
\begin{equation}\label{eq:est1}
|P_t(V(f))|^2 \leq  {\underbrace{\left(|\divv V|_\infty+ \frac{C_2(T)}{\sqrt{t}} |V|_\infty\right)}_{=:\text{ }\alpha_2(t,V)}}^2 P_t f^2
\end{equation}
for all $t \in (0,T]$ and $f \in C^1_b(M)$. These estimates can be used to derive \textit{shift-Harnack inequalities}, as shown by F.-Y. Wang for the case of a Markov operator on a Banach space (see \cite[Proposition~2.3]{wang}). In particular, suppose $\lbrace F_s\colon s \in [0,1] \rbrace$ is a $C^1$ family of diffeomorphisms of $M$ with $F_0 = \id_M$. For each $s \in [0,1]$ define a vector field $V_s$ on $M$ by
\begin{equation}
V_s := (DF_s)^{-1} \dot{F}_s
\end{equation}
and assume $V_s$ and $\divv V_s$ are uniformly bounded. Note $\frac{d}{ds}(f\circ F_s) = \nabla_{V_s} (f\circ F_s)$. Fixing $p \geq 1$ and setting $\beta(s) = 1+(p-1)s$, as in the first part of \cite[Proposition~2.3]{wang}, we deduce from inequality \eqref{eq:est2} that
\begin{equation}
\frac{d}{ds} \log \left( P_t(f^{\beta(s)}\circ F_s)\right)^{p/\beta(s)} \geq - \frac{p}{\beta(s)} \alpha_1\left(\frac{\beta'(s)}{\beta(s)},t,V_s\right)
\end{equation}
for all $s \in [0,1]$, which when integrated gives the shift-Harnack inequality
\begin{equation}
(P_t f)^p \leq \left(P_t(f^p\circ F_1)\right)\exp\left(\int_0^1 \frac{p}{\beta(s)} \alpha_1\left(\frac{\beta'(s)}{\beta(s)},t,V_s\right)ds\right)
\end{equation}
for each $t\in [0,T]$ and positive $f \in C^1_b(M)$. Alternatively, from inequality \eqref{eq:est1} and following the calculation in the second part of \cite[Proposition~2.3]{wang}, we deduce
\begin{equation}
P_t f \leq P_t (f\circ F_1) + \left( \int_0^1 \alpha_2(t,V_s) ds\right)^{{1}/{2}} \sqrt{P_t f^2}
\end{equation}
for each $t\in [0,T]$ and positive $f \in C^1_b(M)$. The shift $F_1$ could be given by the exponential of a well-behaved vector field; the shifts considered in \cite{wang} are of the form $x \mapsto x+v$, for some $v$ belonging to the Banach space.

\section{Extrinsic Formulae}\label{sec:two}

Suppose now that $M$ is simply a smooth manifold of dimension $n$. Suppose
$A_0$ is a smooth vector field and
$$A\colon M \times \R^m \rightarrow TM,\quad (x,e)\mapsto
A(x)e,$$
a smooth bundle map over $M$. This means $A(\newdot)e$ is a vector
field on $M$ for each $e\in\R^m$, and $A(x)\colon\R^m\to T_xM$ is
linear for each $x\in M$

For an $\R^m$-valued Brownian motion $B_t$, sped up by $2$ so
that $d[B,B]_t = 2\,\id_{\R^m}dt$, defined on a filtered
probability space
$(\Omega,\SF,\P;(\SF_t)_{t \in \R_+})$, satisfying the
usual completeness conditions, consider the Stratonovich stochastic
differential equation
\begin{equation}\label{eqn:SDE}
  d X_t = A_0(X_t)\,dt+ A(X_t)\circ dB_t.
\end{equation}
Given an orthonormal basis $\lbrace e_i \rbrace_{i=1}^m$ of
$\R^m$ set $A_i(\newdot) := A(\newdot) e_i$ and
$B^i_t := \langle B_t,e_i\rangle$. Then the previous equation can be
equivalently written
\begin{equation}
  d X_t = A_0(X_t)\,dt+\sum_{i=1}^m A_i(X_t)\circ dB^i_t.
\end{equation}
There is a partial flow $X_t(\newdot)$, $\zeta(\newdot)$ associated to
\eqref{eqn:SDE} (see \cite{kunita} for details) such that for each
$x \in M$ the process $X_t(x)$, $0\leq t<\zeta(x)$ is the maximal
strong solution to \eqref{eqn:SDE} with starting point $X_0(x) = x$,
defined up to the explosion time $\zeta(x)$; moreover using the
notation $X_t(x,\omega) = X_t(x)(\omega)$ and
$\zeta(x,\omega) = \zeta(x)(\omega)$, if
\begin{equation}
  M_t(\omega) = \lbrace x \in M: t< \zeta(x,\omega)\rbrace
\end{equation}
then there exists $\Omega_0 \subset \Omega$ of full measure such that
for all $\omega \in \Omega_0$:
\begin{itemize}
\item[i)] $M_t(\omega)$ is open in $M$ for each $t\geq 0$,
  i.e. $\zeta(\newdot,\omega)$ is lower semicontinuous on $M$;
\item[ii)] $X_t(\newdot,\omega):M_t(\omega)\rightarrow M$ is a
  diffeomorphism onto an open subset of $M$;
\item[iii)] The map $s\mapsto X_s(\newdot,\omega)$ is continuous from
  $[0,t]$ into $C^\infty(M_t(\omega),M)$ with its $C^\infty$-topology,
  for each $t>0$.
\end{itemize}
The solution processes $X = X(x)$ to \eqref{eqn:SDE} are diffusions on
$M$ with generator
\begin{equation}\label{eqn:generator}
  \SL:={A_0} +\sum_{i=1}^m A_i^2.
  % L:=\L_{A_0} + \sum_{i=1}^m \L_{A_i} \L_{A_i}
\end{equation}
% where $\L$ denotes Lie differentiation.
We will assume that the equation is non-degenerate, which is to say
that $A(x):\R^m \rightarrow T_xM$ is surjective for all
$x \in M$. Then $A$ induces a Riemannian metric on $M$, the quotient
metric, with respect to which
\begin{equation}
  A(x)^\ast = (A(x)\vert_{\ker A(x)^\perp})^{-1}
\end{equation}
and whose inner product $\langle \newdot,\newdot \rangle$ on a tangent
space $T_xM$ is given by
\begin{equation}
  \langle v,u\rangle = \langle A(x)^\ast v,A(x)^\ast u\rangle_{\R^m}.
\end{equation}

\subsection{A formula for the differential}

Denote by
\begin{equation}
  P_tf(x):=\E \left[ \1_{\lbrace t < \zeta(x)\rbrace} f(X_t(x))\right]
\end{equation}
the minimal semigroup associated to equation \eqref{eqn:SDE}, acting
on bounded measurable functions $f$. In terms of any linear connection
$\tilde{\nabla}$ on $TM$ with adjoint $\tilde{\nabla}'$ (see \eqref{eq:adjointdefn} below), a solution $TX_{t}(x)$ to the derivative
equation
\begin{equation}
  d^{\tilde{\nabla}'} TX_{t}(x) = \tilde{\nabla}_{TX_{t}(x)}A_0\,dt + \sum_{i=1}^m \tilde{\nabla}_{TX_{t}(x)} A_i \circ dB^i_t
\end{equation}
with $TX_{0}(x) = \id_{T_{x}M}$ is the derivative (in probability) at
$x$ of the solution flow to \eqref{eqn:SDE}. Our objective will be to
find a formula for $P_T (V(f))$ in terms of $TX_t$. Before doing
so, let us briefly derive the corresponding formula for
$(dP_T)(v)$. As in Subsection \ref{ss:formdif}, let $\ell_t$ be an
adapted process with paths in $L^{1,2}([0,T];T_{x}M)$. By It\^{o}'s
formula and the Weitzenb\"{o}ck formula (see \cite[Theorem~2.4.2]{ELJL3}) it follows, according to the procedure of Subsection \ref{ss:formdif}, that
\begin{equation}\label{eq:locmartingone}
  (dP_{T-t}f) (TX_t(x) {\ell}_t) -\frac12(P_{T-t}f)(X_t(x))\int_0^t \langle TX_s(x) \dot{\ell}_s ,A(X_s(x))\,dB_s\rangle
\end{equation}
is a local martingale, starting at $(dP_Tf)(\ell_0)$. Choosing
$\ell_t$ so that \eqref{eq:locmartingone} is a true martingale with
$\ell_0 =v$ and $\ell_T =0$, we obtain the formula
\begin{equation}\label{eq:locformoneform}
  (dP_Tf)(v) = - \frac12\E\left[\1_{\lbrace T < \zeta(x)\rbrace}f(X_T(x)) \int_0^{T} \langle TX_s(x) \dot{\ell}_s,A(X_s(x))\,dB_s\rangle \right].
\end{equation}
This formula is well-known; it is the one given by
\cite[Theorem~2.4]{Thalmaier97}. Formula \eqref{eq:locformzeroform}
can be obtained from it by filtering. Furthermore, it as always
possible to choose such $\ell_t$, as in Subsection
\ref{ss:formdif}. Now denote by $p_t(x,y)$ the smooth heat kernel
associated to \eqref{eqn:SDE} such that
\begin{equation}
  P_tf(x) = \int_M f(y) p_t(x,y) \,\vol(dy)
\end{equation}
where $\vol(dy)$ denotes integration with respect to the induced
Riemannian volume measure. Since formula \eqref{eq:locformoneform}
holds for all smooth functions $f$ of compact support, we deduce from
it the Bismut formula
\begin{equation}
  (d \log p_T(\newdot,y))_x(v) 
= -\frac12\E\left[ \int_0^{\tau \wedge T} \langle TX_s(x) \dot{\ell}_s,A(X_s(x))\,dB_s\rangle\rangle 
\big\vert \, X_T(x) = y \right],
\end{equation}
the original version of which was given in \cite{bismut} for compact
manifolds. The version stated here is
\cite[Corollary~2.5]{Thalmaier97}, the non-local version having been
earlier given in \cite{ElworthyLi}.

\subsection{Induced linear connections}

There are a number of linear connections naturally associated to the
map $A$. Firstly, there is the Levi-Civita connection $\nabla$ for the
induced metric. Secondly, there is the \textit{Le Jan-Watanabe
  connection}, which is given by the push forward under $A$ of the flat
connection on $\R^m$. Its covariant derivative
$\breve{\nabla}$ is defined by
\begin{equation}\label{eq:covder}
  \breve{\nabla}_v U = A(x)d\left(A(\newdot)^\ast U(\newdot)\right)_x (v)
\end{equation}
for a vector field $U$ and $v \in T_xM$. Like the Levi-Civita
connection, it is adapted to the induced metric. In fact, all metric
connections on $TM$ arise in this way. In addition to the properties
of $\breve{\nabla}$ summarized below, further details of it can be found in
\cite{ELJL3,ELJL1,ELJL2}. It has the property that if
$e \in \ker A(x)^\perp$ then $\breve{\nabla}_v A_e = 0$ for all
$v \in T_xM$, where by $A_e$ we mean the section $x \mapsto A(x)e$. It
therefore satisfies the Le Jan-Watanabe property
\begin{equation}
  \sum_{i=1}^m \breve{\nabla}_{A_i}A_i = 0.
\end{equation}
To any linear connection $\tilde{\nabla}$ on $TM$ one can
associate an adjoint connection $\tilde{\nabla}'$ by
\begin{equation}\label{eq:adjointdefn}
  \tilde{\nabla}'_vU=\tilde{\nabla}_v U - \tilde{T}(v,U)
\end{equation}
for $v$ a vector and $U$ a smooth vector field, where $\tilde{T}$
denotes the torsion tensor of $\tilde{\nabla}$. The adjoint of the Le
Jan-Watanabe connection will be denoted by $\hat{\nabla}$. It therefore
satisfies
\begin{equation}
  \hat{\nabla}_vU=\breve{\nabla}_v U - \breve{T}(v,U)
\end{equation}
or equivalently $\breve{\nabla}_vU=\hat{\nabla}_v U - \hat{T}(v,U)$,
where $\breve{T}$ and $\hat{T}$ denote the torsion tensors of
$\breve{\nabla}$ and $\hat{\nabla}$, respectively; these antisymmetric
tensors satisfy $\breve{T} = -\hat{T}$. By
\cite[Proposition~2.2.3]{ELJL3} the torsion can be written in terms of
$A$ by
\begin{equation}\label{eqn:torinX}
  \breve{T}(v,u)_x = A(x) (dA^\ast)_x(v,u)
\end{equation}
where $dA^\ast$ denotes the exterior derivative of the
$\R^m$-valued $1$-form $A^\ast:TM
\rightarrow\R^m$.
The adjoint connection can therefore be written in terms of $A$ by
\begin{equation}\label{eqn:adjinX}
  \hat{\nabla}_v U 
  = A(x)\left(d\left(A^\ast(\newdot) U(\newdot)\right)_x (v) - (dA^\ast)_x(v,U)\right).
\end{equation}
Besides torsion, we will also encounter several expressions involving
curvature, including
\begin{equation}\label{eqn:Ric}
  \begin{split}
    \breve{{\Ric}} := \sum_{i=1}^m \breve{R}(\newdot,A_i)A_i
  \end{split}
\end{equation}
where $\breve{R}$ denotes the curvature tensor of $\breve{\nabla}$. In
particular, \cite[Lemma~2.4.3]{ELJL3} states for a smooth $1$-form
$\phi$ that
\begin{equation}\label{eq:usingRic}
  \sum_{i=1}^m \L_{A_i}\L_{A_i}\phi = \trace \hat{\nabla}^2 \phi - \phi(\breve{\Ric})
\end{equation}
where $\L$ denotes Lie differentiation.

\subsection{Induced differential operators}

With respect to the metric induced by $A$, we set $\delta :=
d^\ast$. For a $1$-form $\phi$, the codifferential $\delta$ satisfies
\begin{equation}\label{eqn:deltaphi}
  \delta \phi = -\sum_{i=1}^m(\nabla_{A_i} \phi)( A_i)
\end{equation}
but this relation does not hold with $\nabla$ replaced by
$\hat{\nabla}$. Nonetheless, for the divergence of a smooth vector
field $U$ we do have
\begin{equation}\label{eqn:divU}
  \divv U = \sum_{i=1}^m\langle \nabla_{A_i} U,A_i \rangle = \sum_{i=1}^m\langle \hat{\nabla}_{A_i} U,A_i \rangle = \trace \hat{\nabla} U
\end{equation}
by the adaptedness of $\breve{\nabla}$.

\begin{lemma}\label{lem:newcom}
  For any smooth vector field $U$, $1$-form $\phi$ and linear
  connection $\tilde{\nabla}$ with adjoint
  $ \tilde{\nabla}'_\bull U = \tilde{\nabla}_\bull U -
  \tilde{T}(\nbull,U)$ we have
  \begin{equation}\label{eqn:divcommute}
    (U+\divv U)\delta \phi = -\delta \left(\tilde{\nabla}^\ast_U+  (\tilde{\nabla}' U)^{\trans \ast}\right)\phi.
  \end{equation}
\end{lemma}

\begin{proof}
  As a linear connection, $\tilde{\nabla}$ satisfies
  \begin{equation}
    \L_U \phi = \tilde{\nabla}_U \phi + \phi(\tilde{\nabla}' U).
  \end{equation}
  Since $d$ commutes with Lie differentiation, we thus have
  \begin{equation}
    d U f = \L_U df = \tilde{\nabla}_U df + df(\tilde{\nabla}' U)= \tilde{\nabla}_U df + (\tilde{\nabla}' U)^{\trans} df.
  \end{equation}
  By duality this implies
  \begin{equation}
    U^\ast \delta \phi = \delta \left(\tilde{\nabla}^\ast_U + (\tilde{\nabla}' U)^{\trans \ast}\right)\phi
  \end{equation}
  and therefore
  \begin{equation}
    (U+\divv U)\delta \phi = -\delta \left(\tilde{\nabla}^\ast_U+  (\tilde{\nabla}' U)^{\trans \ast}\right)\phi
  \end{equation}
  since $U^\ast = -U -\divv U$.
\end{proof}

With respect to the induced metric, the formal adjoint $\nabla_U^\ast$
of the differential operator $\nabla_U$ acting on $1$-forms is given
by
\begin{equation}
  \nabla_U^\ast = -\nabla_U-\divv U.
\end{equation}
More generally, we have the following lemma.

\begin{lemma}\label{lem:newcom2}
  For any smooth vector field $U$ and metric connection
  $\tilde{\nabla}'$ with adjoint $\tilde{\nabla}$ we have
  \begin{equation}
    \tilde{\nabla}^\ast_U = - \tilde{\nabla}_U - \divv U - \tilde{T}(U,\cdot)-\tilde{T}(U,\cdot)^\ast.
  \end{equation}
\end{lemma}

\begin{proof}
  Denoting by $\mu_g$ the Riemannian volume density, the divergence of
  a vector field $U$ satisfies $\L_U \mu_g = (\divv U) \mu_g$ and thus
  for compactly supported $1$-forms $\phi,\psi$ we have
  \begin{equation}
    \begin{split}
      \L_U(\langle \phi,\psi\rangle \mu_g) &= \langle \tilde{\nabla}'_U \phi,\psi \rangle \mu_g + \langle  \phi,\tilde{\nabla}'_U \psi \rangle \mu_g + (\divv U) \langle \phi, \psi\rangle \mu_g\\
      &= \langle \tilde{\nabla}_U \phi,\psi \rangle \mu_g + \langle  \phi,\tilde{\nabla}_U \psi \rangle \mu_g+ \langle \phi(\tilde{T}(U,\newdot)),\psi\rangle\mu_g \\
      &\quad + \left\langle
        \phi,\psi(\tilde{T}(U,\newdot))\right\rangle\mu_g +(\divv U)
      \langle \phi, \psi\rangle \mu_g
    \end{split}
  \end{equation}
  from which the result follows, since $\int_M \L_U(\langle \phi,\psi\rangle \mu_g) = 0$, by Stokes'
  theorem.
\end{proof}

The map $A$ also induces a differential operator $\hat{\delta}$,
mapping $1$-forms to functions by
\begin{equation}
  \hat{\delta} \phi := - \sum_{i=1}^m\iota_{A_i} \L_{A_i}\phi.
\end{equation}
Since $\L_{A_i} \phi = \iota_{A_i}d\phi + d(\iota_{A_i} \phi)$, the
generator $\SL$ can be expressed in terms of $\hat{\delta}$ by
\begin{equation}\label{eq:genrel}
  \SL = \L_{A_0}-(\hat{\delta}d + d \hat{\delta}).
\end{equation}
Clearly $\hat{\delta}^2=0$, so to find an analogue of the second
commutation rule in Lemma \ref{Lemma:CommRules} for $\hat{\delta}$ and
$\SL$ it suffices to calculate the Lie derivative of $\hat{\delta}$ in
the direction $A_0$. This is the main objective of the remainder of
this section. Note that $\hat{\delta}$ need not agree with the
codifferential $\delta$. For any smooth vector field $U$ and linear
connection $\tilde{\nabla}$ with adjoint $\tilde{\nabla}'$ we have
\begin{equation}\label{eqn:generadj}
  \L_U\phi = (\tilde{\nabla}_U \phi) + \phi(\tilde{\nabla}'U)
\end{equation}
and therefore
\begin{equation}\label{eqn:hatbreve}
  \hat{\delta} \phi = - \sum_{i=1}^m(\hat{\nabla}_{A_i} \phi)({A_i}) -\sum_{i=1}^m \phi(\breve{\nabla}_{A_i}{A_i}) = - \sum_{i=1}^m(\hat{\nabla}_{A_i} \phi)({A_i})
\end{equation}
or alternatively
\begin{equation}\label{eqn:brevehat}
  \hat{\delta} \phi = - \sum_{i=1}^m(\breve{\nabla}_{A_i} \phi)({A_i}) -\sum_{i=1}^m \phi(\hat{\nabla}_{A_i}{A_i}) = - \sum_{i=1}^m(\breve{\nabla}_{A_i} \phi)({A_i})
\end{equation}
by the Le Jan-Watanabe property and the fact that
$\breve{T}(A_i,A_i)=0$. Applying \eqref{eqn:generadj} to the
Levi-Civita connection gives
\begin{equation}
  \hat{\delta} \phi = - \sum_{i=1}^m({\nabla}_{A_i} \phi)({A_i}) - \sum_{i=1}^m\phi({\nabla}_{A_i}{A_i})
\end{equation}
and so by \eqref{eqn:deltaphi} we have
\begin{equation}\label{eq:relation}
  \hat{\delta} \phi = \delta \phi - \phi\left(\sum_{i=1}^m{\nabla}_{A_i}{A_i}\right)
\end{equation}
which expresses the difference of the operators $\delta$ and
$\hat{\delta}$.

\begin{lemma}\label{lem:finalthing}
  For any smooth vector field $U$ and $1$-form $\phi$ we have
  \begin{equation}
    \left(U +\trace \hat{\nabla}U\right)\hat{\delta}\phi = \hat{\delta}\left(\hat{\nabla}_U - (\breve{\nabla}U)^{\trans \ast}- \breve{T}(U,\cdot)-\breve{T}(U,\cdot)^\ast +\trace \hat{\nabla}U\right)\phi+\phi(U^A)
  \end{equation}
  where the vector field $U^A$ is defined by
  \begin{equation}
    U^A:=-\sum_{i=1}^m \left((\hat{\nabla}U)^{ \ast} +\hat{\nabla} U\right)(\nabla_{A_i}A_i)-\sum_{i=1}^m [U,\nabla_{A_i}A_i].
  \end{equation}
\end{lemma}

\begin{proof}
  By Lemmas \ref{lem:newcom} and \ref{lem:newcom2} we have
  \begin{equation}
    (U+ \divv U)\delta \phi = \delta(\hat{\nabla}_U + \divv U + \hat{T}(U,\cdot)+\hat{T}(U,\cdot)^\ast - (\breve{\nabla}U)^{\trans \ast})\phi.
  \end{equation}
  By \eqref{eq:relation} we have
  \begin{equation}
    \begin{split}
      (U+ \divv U)\delta \phi &= (U+\divv U)\hat{\delta}\phi + (\divv U) \phi(\nabla_{A_i}A_i)\\
      &\quad + (\hat{\nabla}_U \phi)(\nabla_{A_i}A_i) + \phi
      (\hat{\nabla}_U \nabla_{A_i}A_i)
    \end{split}
  \end{equation}
  and
  \begin{equation}
    \begin{split}
      &\delta\left(\hat{\nabla}_U + \divv U + \hat{T}(U,\cdot)+\hat{T}(U,\cdot)^\ast - (\breve{\nabla}U)^{\trans \ast}\right)\phi \\
      =\text{ }&\hat{\delta}\left(\hat{\nabla}_U + \divv U - \breve{T}(U,\cdot)-\breve{T}(U,\cdot)^\ast - (\breve{\nabla}U)^{\trans \ast}\right)\phi + (\hat{\nabla}_U \phi)(\nabla_{A_i}A_i)\\
      &+ \left((\divv U - \breve{T}(U,\cdot)-\breve{T}(U,\cdot)^\ast -(\breve{\nabla}U)^{\trans
          \ast})\phi\right)(\nabla_{A_i}A_i).
    \end{split}
  \end{equation}
  Rearranging, the result follows by equation \eqref{eqn:divU}.
\end{proof}

Note that the vector field $A_0^A$ appears to depend on the
Levi-Civita connection via the sum of the vector fields
$\nabla_{A_i}A_i$. It is clear that all other objects appearing in the
definition of $A_0^A$ can be calculated explicitly in terms of $A$ and
$A_0$, by formula \eqref{eq:covder}. The following lemma, combined
with formula \eqref{eqn:torinX}, shows that the sum of the vector
fields $\nabla_{A_i} A_i$ can also be expressed directly in terms of
$A$.

\begin{lemma}\label{lem:contor}
  We have
  \begin{equation}
    \sum_{i=1}^m \nabla_{A_i}{A_i} = - \sum_{i=1}^m  \breve{T}(\newdot,A_i)^\ast (A_i)
  \end{equation}
  where $\breve{T}$ denotes the torsion of the Le Jan-Watanabe
  connection.
\end{lemma}

\begin{proof}
  Suppressing the summation over $i$, the Le Jan-Watanabe property
  implies
  \begin{equation}\label{eqn:nablaxx}
    \nabla_{A_i}{A_i} = \breve{\nabla}_{A_i}{A_i} - \breve{K}(A_i,A_i) = -\breve{K}(A_i,A_i)
  \end{equation}
  where $\breve{K}$ denotes the contorsion tensor of
  $\breve{\nabla}$. The contorsion tensor measures the extent to which
  a metric connection fails to be the Levi-Civita connection,
  vanishing if the connection is torsion free. It is discussed in
  \cite{K-N} and \cite{Nakahara}. The components of $\breve{K}$
  satisfy $\tensor{\breve{K}}{^i_j_j} = \tensor{\breve{T}}{_j^i_j}$,
  which is to say
  \begin{equation}
    \breve{K}(A_i,A_i) = \big(\breve{T}(\newdot,A_i)^\flat\big)(A_i)^\sharp,
  \end{equation}
  where $\flat$ and $\sharp$ are the musical isomorphisms associated
  to the induced metric. This implies
  \begin{equation}
    \langle \breve{K}(A_i,A_i),U\rangle = \langle \breve{T}(U,A_i),A_i\rangle
  \end{equation}
  for all smooth vector fields $U$, and therefore
  \begin{equation}\label{eqn:kxx}
    \breve{K}(A_i,A_i) = \breve{T}(\newdot,A_i)^\ast (A_i)
  \end{equation}
  as required.
\end{proof}

Consequently
\begin{equation}\label{eq:A0def}
  A_0^A=\sum_{i=1}^m \left(\hat{\nabla} A_0 + (\hat{\nabla}A_0)^{ \ast}\right)\left(\breve{T}(\newdot,A_i)^\ast (A_i)\right)+[A_0,\breve{T}(\newdot,A_i)^\ast (A_i)].
\end{equation}

\subsection{Commutation formula}

We have, in summary, the following commutation rule, extending formula
\eqref{eq:comrulediv}.

\begin{prop}\label{prop:finalthmcom}
  For any smooth $1$-form $\phi$ we have

 \begin{equation}
 \begin{split}
    &\left(\SL+\trace\hat{\nabla}A_0\right)\hat{\delta} \phi \\
    &\quad\quad\quad=\text{ } \hat{\delta}\left( \trace \hat{\nabla}^2 +\hat{\nabla}_{A_0} -\breve{\Ric}_{A_0}-\hat{\nabla}{A_0} - (\hat{\nabla}{A_0})^{\ast\trans}+\trace \hat{\nabla}A_0\right)\phi+\phi(A_0^A)
    \end{split}
  \end{equation}  
  
  where the vector field $A_0^A$ is given by \eqref{eq:A0def} and $\breve{\Ric}_{A_0}:= \breve{\Ric} -\breve{\nabla}A_0$.
\end{prop}

\begin{proof}
  The claim follows from Lemmas \ref{lem:finalthing} and
  \ref{lem:contor} and the relations \eqref{eq:usingRic} and
  \eqref{eq:genrel}.
\end{proof}

Finally, note that for a smooth function $f$, the codifferential
$\delta$ satisfies
\begin{equation}\label{eqn:decomplc}
  \langle df,\phi\rangle = f\delta(\phi)-\delta(f\phi).
\end{equation}
We will need an analogous formula for $\hat{\delta}$, as given by the
following lemma.

\begin{lemma}\label{lem:decomp}
  For any smooth function $f$ we have
  \begin{equation}
    \langle df,\phi\rangle = f\hat{\delta}(\phi)-\hat{\delta}(f\phi).
  \end{equation}
\end{lemma}

\begin{proof}
  Suppressing notationally the summation over $i$, we have
  \begin{equation}
    \begin{split}
      \hat{\delta}(f\phi)&= - \iota_{A_i} \L_{A_i} (f \phi)\\
      &= - \iota_{A_i} \left(\iota_{A_i} d(f\phi) + d(\iota_{A_i}f \phi)\right)\\
      &=-\iota_{A_i}\left(\iota_{A_i}(df \wedge \phi + f d \phi )+ \phi(A_i)df + f d(\phi(A_i))\right)\\
      &=-\iota_{A_i}\iota_{A_i}(df \wedge \phi) - \phi(A_i)df(A_i) + f\hat{\delta}(\phi)\\
      &= - \langle df,\phi\rangle + f\hat{\delta}(\phi)
    \end{split}
  \end{equation}
  since $\iota_{A_i}\iota_{A_i}(df \wedge \phi) = 0$.
\end{proof}

Now we are in a position to deduce formulae for the induced
differential operator in terms of the derivative flow $TX_t$.

\subsection{A formula for the induced differential operator}

We must now assume equation \eqref{eqn:SDE} is complete, which is to
say $\zeta(x) = \infty$, almost surely. For a bounded smooth $1$-form
$\alpha$ suppose $\alpha_t$ satisfies
\begin{equation}
  \partial_t \alpha_t = (\trace \hat{\nabla}^2+\hat{\nabla}_{A_0} - {\breve{\Ric}}_{A_0}-\hat{\nabla} A_0- (\hat{\nabla} A_0)^{\ast\trans}+\trace \hat{\nabla}A_0)\alpha_t
\end{equation}
with $\alpha^{\mstrut}_0 = \alpha$ and that
$\Xi_t(x):T_{x}M \rightarrow T_{X_t(x)}M$ solves the covariant It\^{o}
equation
\begin{equation}
  d^{\hat{\nabla}}\Xi_t(x) = - \left( {\breve{\Ric}}_{A_0}+\hat{\nabla} A_0+ (\hat{\nabla} A_0)^{\ast\trans}+\trace \hat{\nabla}A_0\right)(\Xi_t(x))\,dt + \sum_{i=1}^m \breve{\nabla}_{\Xi_t(x)}A_i\,dB_t^j 
\end{equation}
along the paths of $X_t(x)$ with $\Xi_0 =\id_{T_{x}M}$. Fixing
$T>0$, by It\^{o}'s formula we have
\begin{equation}\label{eq:itoformap}
  \begin{split}
    d(\alpha^{\mstrut}_{T-t}(\Xi_t(x))) &= \sum_{i=1}^m\hat{\nabla}_{A_i} \alpha^{\mstrut}_{T-t}(\Xi_t(x)) \,dB^i_t + \hat{\nabla}_{A_0}\alpha^{\mstrut}_{T-t}(\Xi_t(x))\,dt + \partial_t \alpha^{\mstrut}_{T-t}(\Xi_t(x))\,dt \\
    &\quad+ \trace \hat{\nabla}^2 \alpha^{\mstrut}_{T-t}(\Xi_t(x))\,dt + \alpha^{\mstrut}_{T-t}(d^{\hat{\nabla}}\Xi_t(x))\\
    &=\sum_{i=1}^m ((\hat{\nabla}_{A_i} \alpha^{\mstrut}_{T-t})
    \newdot + \alpha^{\mstrut}_{T-t}(\breve{\nabla}_{\bull}
    A_i))(\Xi_t(x)) \,dB^i_t.
  \end{split}
\end{equation}
It follows that $\alpha^{\mstrut}_{T-t}(\Xi_t(x))$ is a local
martingale, starting at $\alpha^{\mstrut}_T$. Furthermore, according
to equation (26) in \cite{ELJL1}, for the derivative process $TX_t(x)$
we have
\begin{equation}
  d^{\hat{\nabla}}TX_{t}(x) =- \breve{\Ric} (TX_{t}(x))\,dt + \breve{\nabla}_{TX_{t}(x)}A_0 \,dt +  \sum_{i=1}^m \breve{\nabla}_{TX_{t}(x)}A_i\,dB_t^i 
\end{equation}
and therefore, by the variation of constants formula, we have
\begin{equation}
  \Xi_t(x) = TX_{t}(x) - TX_{t}(x) \int_0^t TX_{s}(x)^{-1}  \left(\big(\hat{\nabla}A_0+
  (\hat{\nabla}A_0)^\ast +\trace \hat{\nabla}A_0\big)(\Xi_s(x))\right)ds.
\end{equation}
Thus it is possible to calculate $\Xi_t(x)$ without using the parallel
transport implicit in the original equation. Moreover, if the vector
field $A_0$ vanishes then $\Xi_t(x)$ is given precisely by the
derivative process $TX_t(x)$.

\begin{prop}\label{prop:thelocm}
  Suppose $h_t$ is an adapted process with paths in
  $L^{1,2}([0,T];\R)$. Then
  \begin{equation}\label{eqn:thelocalmart}
    \begin{split}
      \hat{\delta}&\alpha^{\mstrut}_{T-t} h_t - \int_0^t h_s \alpha^{\mstrut}_{T-s}(A_0^A)\,ds\\
      &+\frac12\alpha^{\mstrut}_{T-t}\left(\Xi_t(x) \int_0^t
        \left(\dot{h}_s -(\trace
          \hat{\nabla}A_0)(X_s(x))h_s\right){\Xi_s(x)^{-1}}
        A(X_s(x))\,dB_s\right)
    \end{split}
  \end{equation}
  is a local martingale, starting at
  $\hat{\delta} \alpha^{\mstrut}_T h_0$, where the vector field
  $A_0^A$ is given by \eqref{eq:A0def}.
\end{prop}

\begin{proof}
  Set
  \begin{equation}
    \A_t := \exp \left(\int_0^t (\trace \hat{\nabla}A_0)(X_s(x))\,ds\right)
  \end{equation}
  and define $\ell_t := \A^{-1}_t h_t$. By equation
  \eqref{eq:itoformap}, integration by parts and formula
  \eqref{eqn:hatbreve}, we have, suppressing the summation over $i$,
  that
  \begin{equation}\label{eqn:alocalmart}
    \begin{split}
      d&\left( \alpha^{\mstrut}_{T-t}\Big(\Xi_t(x) \frac12 \int_0^t \A_s\dot{\ell}_s{\Xi_s(x)^{-1}} A(X_s(x))\,dB_s\Big)\right) \\
      &\mequal \frac12\left(\big((\hat{\nabla}_{A_i} \alpha^{\mstrut}_{T-t})\,\nbull + \alpha^{\mstrut}_{T-t}(\breve{\nabla}_{\bull} A_i)\big)(\Xi_t(x)) \,dB^i_t\right)\left( \A_t\dot{\ell}_t{\Xi_t(x)^{-1}} A_j(X_t(x))\,dB^j_t\right)\\
      &=\left((\hat{\nabla}_{A_i} \alpha^{\mstrut}_{T-t}) A_i + \alpha^{\mstrut}_{T-t}(\breve{\nabla}_{A_i} A_i)\right)\A_t\dot{\ell}_t\,dt\\
      &=(\hat{\nabla}_{A_i} \alpha^{\mstrut}_{T-t}) A_i \A_t\dot{\ell}_t\,dt\\
      &=-(\hat{\delta} \alpha^{\mstrut}_{T-t})\A_t\dot{\ell}_t\,dt
    \end{split}
  \end{equation}
  where $\mequal$ denotes equality modulo the differential of a local
  martingale. By Proposition \ref{prop:finalthmcom} and It\^{o}'s
  formula we have
  \begin{equation}
    d( \A_t \hat{\delta} \alpha^{\mstrut}_{T-t}) \mequal \A_t \hat{\delta}\partial_t \alpha^{\mstrut}_{T-t} \,dt  + \A_t(\SL+\trace \hat{\nabla}A_0) \hat{\delta}\alpha^{\mstrut}_{T-t} \,dt =\A_t\alpha^{\mstrut}_{T-t}(A_0^A)\,dt
  \end{equation}
  which implies
  \begin{equation}
    n_t := \A_t \hat{\delta} \alpha^{\mstrut}_{T-t} -\int_0^t \A_s\alpha^{\mstrut}_{T-s} (A_0^A)ds
  \end{equation}
  is a local martingale, starting at
  $\hat{\delta}\alpha^{\mstrut}_T$. This implies
  \begin{equation}
    \begin{split}
      d(n_t \ell_t ) &\mequal n_t \dot{\ell}_t\,dt\\
      &=
      (\hat{\delta}\alpha^{\mstrut}_{T-t})\A_t\dot{\ell}_t\,dt-\dot{\ell}_t\int_0^t\A_s\alpha^{\mstrut}_{T-s}
      (A_0^A)\,ds dt.
    \end{split}
  \end{equation}
  Substituting the definition of $n_t$ into the left-hand side and
  performing integration by parts to the second term on the right-hand
  side implies
  \begin{equation}\label{eqn:peneq}
    \hat{\delta}\alpha^{\mstrut}_{T-t}h_t - \int_0^t (\hat{\delta} \alpha^{\mstrut}_{T-s})\A_s\dot{\ell}_s \,ds- \int_0^t h_s \alpha^{\mstrut}_{T-s}(A_0^A)\,ds 
  \end{equation}
  is another local martingale. Since
  \begin{equation}
    \dot{\ell}_t = \A_t^{-1}\left(\dot{h}_t -(\trace \hat{\nabla}A_0)(X_t(x))h_t\right),
  \end{equation}
  substituting formula \eqref{eqn:alocalmart} into the second term in
  \eqref{eqn:peneq} completes the proof.
\end{proof}

\begin{thm}\label{thm:divsemi}
  Suppose $h_t$ is any adapted process with paths in
  $L^{1,2}([0,\infty);\R)$ such that $h_0=0$ and $h_T=1$ and
  that $\alpha$ is a bounded smooth $1$-form. Suppose $\eqref{eqn:SDE}$
  is complete and that the local martingales
  $\alpha^{\mstrut}_{T-t}(\Xi_t)$ and \eqref{eqn:thelocalmart} are
  true martingales. Then
  \begin{equation}
    P_T (\hat{\delta}\alpha) =  -\frac12\E\left[\alpha\left(\Xi_T \int_0^T {\Xi_t^{-1}} \left( \left(\dot{h}_t-(\trace \hat{\nabla}A_0)(X_t)h_t\right) A(X_t)\,dB_t + 2 h_t A_0^A \,dt\right)\right)\right].
  \end{equation}
\end{thm}

\begin{proof}
  By \eqref{eq:itoformap} we have
  \begin{equation}
    \alpha^{\mstrut}_{T-t}(\Xi_t) = \alpha(\Xi_T)-\int_t^T \left(\big(\hat{\nabla}_{A_i} \alpha^{\mstrut}_{T-s}\big) \,\nbull + \alpha^{\mstrut}_{T-s}(\breve{\nabla}_{\bull} A_i)\right)(\Xi_t) \,dB^i_t
  \end{equation}
  and therefore
  \begin{equation}
    \E\left[\int_0^T  \alpha^{\mstrut}_{T-t} (\Xi_t  h_t \Xi_t^{-1} A_0^A) \,dt\right] =\E\left[ \alpha\left(\Xi_T \int_0^T h_t \Xi_t^{-1} A_0^A \,dt\right) \right]
  \end{equation}
  since $\alpha^{\mstrut}_{T-t}(\Xi_t)$ is assumed to be a
  martingale. The result now follows from Proposition
  \ref{prop:thelocm}, by taking expectations.
\end{proof}

In analogy with Lemma \ref{lem:truemart}, an integrability assumption
on $h$ plus suitable bounds on
$\hat{\nabla}A_0$,
$\trace \hat{\nabla} A_0$, $A_0^A$ and $\hat{\delta} \alpha$ and on
the moments of $TX_t$ and $TX_t^{-1}$ would be sufficient to guarantee that
$\alpha^{\mstrut}_{T-t}(\Xi_t)$ and \eqref{eqn:thelocalmart} are true
martingales.

\begin{cor}\label{cor:maincor}
  Suppose $f$ is a bounded smooth function. Suppose $V$ is a bounded smooth
  vector field with $\sum_{i=1}^m A_i\langle V,A_i\rangle$
  bounded. Then, under the assumptions of Theorem \ref{thm:divsemi}
  with $\alpha = fV^\flat$, we have
    \begin{equation}
    \begin{split}
      &P_T(V(f)) = -\sum_{i=1}^m \E\left[ f(X_T)\,A_i\langle V,A_i\rangle (X_T)\right] \\
      &+\frac12\E\left[ f(X_T) \,\bigg\langle V\,
        ,\,\Xi_T \int_0^T {\Xi^{-1}_{t}}
        \left(\left(\dot{h}_t-(\trace\hat{\nabla}A_0)(X_t)h_t \right)A(X_t)\,dB_t
          + 2h_t A_0^A \,dt\right)\bigg \rangle\right]
    \end{split}
  \end{equation}
    with
  \begin{equation}
    \begin{split}
      \Xi_t &= TX_{t} - TX_{t} \int_0^t TX_{s}^{-1} \left((\hat{\nabla}A_0)^\ast + \hat{\nabla}A_0+ \trace \hat{\nabla}A_0\right)(\Xi_s)\,ds,\\
      A_0^A&=\sum_{i=1}^m\left((\hat{\nabla}A_0)^{ \ast}
        +\hat{\nabla} A_0\right)\left(\breve{T}(\newdot,A_i)^\ast
        (A_i)\right)+\big[A_0,\breve{T}(\newdot,A_i)^\ast (A_i)\big],
    \end{split}
  \end{equation}
  where the operators $\hat{\nabla} A_0$ and
  $\breve{T}(\newdot,A_i)$ are given at each $x \in M$ and
  $v \in T_xM$ by
  \begin{equation}
    \begin{split}
      \hat{\nabla}_v A_0 &= A(x)\left(d\left(A^\ast(\newdot) A_0(\newdot)\right)_x (v) - (dA^\ast)_x(v,A_0)\right),\\
      \breve{T}(v,A_i)_x &= A(x) (dA^\ast)_x(v,A_i).
    \end{split}
  \end{equation}
\end{cor}

\begin{proof}
  This follows from Theorem \ref{thm:divsemi}. In particular, Lemma
  \ref{lem:decomp} implies
  \begin{equation}
    V(f) = f \hat{\delta} (V^{\flat}) - \hat{\delta}(f V^{\flat})
  \end{equation}
  while formula \eqref{eqn:brevehat}, the Le Jan-Watanabe property and
  the adaptedness of $\breve{\nabla}$ imply
  \begin{equation}
    \hat{\delta} (V^\flat) = -\sum_{i=1}^m\langle \breve{\nabla}_{A_i} V,A_i\rangle = -\sum_{i=1}^mA_i\langle V,A_i\rangle.\qedhere
  \end{equation}
\end{proof}

Note that if \eqref{eqn:SDE} is a gradient system then
$\SL=\Delta+ A_0$ and $A_0^A$ vanishes and
\begin{equation}
  \sum_{i=1}^m A_i\langle V,A_i\rangle = \divv V.
\end{equation}
In this case, since $\trace \hat{\nabla} A_0 = \divv A_0$, Corollary
\ref{cor:maincor} yields the unfiltered version of Corollary
\ref{cor:formula_cor}.

\begin{cor}\label{cor:reversebis}
  Under the assumptions of Corollary \ref{cor:maincor} we have
    \begin{equation}
    \begin{split}
      &\big(d \log p_T(x,\newdot)\big)_y(v) = -\Big\langle v,\sum_{i=1}^m \breve{T}(\newdot,A_i)^\ast (A_i) (y)\Big\rangle\\
      & -\frac12\Big\langle v,\E\left[\Xi_T \int_0^T
        {\Xi^{-1}_{t}} \left(\left(\dot{h}_t-(\trace
            \hat{\nabla}A_0)(X_t)h_t\right)A(X_t)\,dB_t +
          2h_t A_0^A \,dt\right) \, \bigg\vert X_T(x) = y
      \right]\Big\rangle
    \end{split}
  \end{equation}
    for all $v \in T_y M$ where the various terms appearing in the
  right-hand side can be calculated as in Corollary \ref{cor:maincor}.
\end{cor}

\begin{proof}
  Since Corollary \ref{cor:maincor} holds for all smooth functions $f$
  and vector fields $V$ of compact support, and since by Lemma
  \ref{lem:decomp}
  \begin{equation}
    f \hat{\delta} (V^{\flat}) - \hat{\delta}(f V^{\flat}) = V(f) = f \delta (V^{\flat}) - \delta(f V^{\flat}),
  \end{equation}
  the result follows from equation \eqref{eq:relation}, Lemma \ref{lem:contor} and Corollary \ref{cor:maincor}.
\end{proof}

\begin{example}
  Consider the special case in which $M= \R^n$. Denote by $q_T(x,y)$
  the smooth density of $X_T(x)$ with respect to the standard
  $n$-dimensional Lebesgue measure. Recall that $p_T(x,y)$ denotes the
  density with respect to the induced Riemannian measure. It follows
  that
  \begin{equation}
    q_T(x,y) = p_T(x,y) \,\rho^{{1}/{2}}(y)
  \end{equation}
  where $\rho(y)$ denotes the absolute value of the determinant of the
  matrix
  \begin{equation}
    \big\lbrace\langle A^\ast \partial_i, A^\ast \partial_j \rangle_{\R^m}^\mstrut (y)\big\rbrace_{i,j=1}^n
  \end{equation}
  in which $\lbrace \partial_i \rbrace_{i=1}^n$ denotes the standard
  basis of vector fields on $\R^n$. Consequently
  \begin{equation}
    \left(d \log q_T(x,\newdot)\right)_y(v) = \left(d \log p_T(x,\newdot)\right)_y(v) 
    + \big(d \log \rho^{{1}/{2}}(\newdot)\big)_y(v)
  \end{equation}
  with the first term on the right-hand side given, in terms of the
  induced metric, by Corollary \ref{cor:reversebis}.
\end{example}

\section*{Acknowledgements}
This work has been supported by the Fonds National de la Recherche Luxembourg (FNR) under the OPEN scheme (project GEOMREV O14/7628746)

\end{document}